\documentclass[12pt]{article}
\usepackage{amsmath,amssymb,amsfonts,amsthm,indentfirst,setspace}
\usepackage{latexsym,bm}
\usepackage{mathrsfs,amsmath,amssymb,amsthm}

\numberwithin{equation}{section}

\theoremstyle{plain}
\newtheorem{theorem}{Theorem}[section]

\newtheorem{lemma}{Lemma}[section]

\newcommand{\be}{\begin{equation}}
\newcommand{\ee}{\end{equation}}
\newcommand{\bea}{\begin{eqnarray}}
\newcommand{\eea}{\end{eqnarray}}
\newcommand{\eeas}{\end{eqnarray*}}
\newcommand{\beas}{\begin{eqnarray*}}
\usepackage{amscd}

\numberwithin{equation}{section}

\makeatletter
\newcommand{\trace}{\mathop{\operator@font Trace}}
\newcommand{\vspan}{\mathop{\operator@font Span}}
\newcommand{\Int}{\mathop{\operator@font Int}}
\newcommand{\grad}{\mathop{\operator@font grad}}
\newcommand{\diver}{\mathop{\operator@font div}}
\newcommand{\Ad}{\mathop{\operator@font Ad}}
\newcommand{\id}{\mathop{\operator@font id}}

\newcommand{\spec}{{\mathop{\operator@font Spec}}(\mathcal H)}

\makeatother

\begin{document}

\title{$\mathfrak D^\perp$-invariant real hypersurfaces in complex Grassmannians of rank two
}

\author{Ruenn-Huah Lee and Tee-How {Loo}\\
Institute of Mathematical Sciences \\
University of Malaya \\
50603 Kuala Lumpur, Malaysia.	\\ 
\ttfamily{rhlee063@hotmail.com },
\ttfamily{looth@um.edu.my}}

\date{}
\maketitle

\abstract{
Let $M$ be a real hypersurface in complex Grassmannians of rank two. Denote by $\mathfrak J$ the quaternionic K\"{a}hler structure of the ambient space,  $TM^\perp$ the normal bundle over $M$ and $\mathfrak D^\perp=\mathfrak JTM^\perp$.  
The real hypersurface $M$ is said to be $\mathfrak D^\perp$-invariant if $\mathfrak D^\perp$ is invariant under the shape operator of $M$.
We showed that if $M$ is $\mathfrak D^\perp$-invariant, then $M$ is Hopf. 
This improves the results of  Berndt and Suh in 
[{Int. J. Math.} \textbf{23}(2012) 1250103] and [{Monatsh. Math.} \textbf{127}(1999), 1--14].
We also classified $\mathfrak D^\perp$ real hypersurface in complex Grassmannians of rank two with constant principal curvatures.
}




\section{Introduction}

Denote by   $\hat M^m(c)$ the compact complex Grassmannian $ SU_{m+2}/S(U_2U_m)$ of rank two 
 (resp.  noncompact complex Grassmannian $SU_{2,m}/S(U_2U_m)$ of  rank two) for $c>0$ (resp. $c<0$), 
where $c=\max||K||/8$ is a scaling factor for the Riemannian metric $g$ and $K$ is the sectional curvature for $\hat M^m(c)$.
It is well known that $\hat M^m(c)$ is a Riemannian symmetric spaces equipped with a K\"{a}hler structure $J$ and a quaternionic K\"{a}hler structure $\mathfrak J$.  
A tangent vector $X\in T_x\hat M^m(c)$, $x\in\hat M^m(c)$ is said to be \emph{singular} if either 
$JX\in\mathfrak JX$ or $JX\perp\mathfrak JX$,

Let $M$ be a connected real hypersurface in $\hat M^m(c)$. Then the K\"ahler structure $J$ and quarternionic K\"ahler structure $\mathfrak J$ naturally induce two subbundles $JTM^\perp$ and $\mathfrak JTM^\perp$ in the tangent bundle $TM$ over $M$. 
Denote by $A$ the shape operator on $M$.
In   \cite{berndt-suh} and  \cite{berndt-suh2},
 Berndt and Suh studied real hypersurfaces $M$ in $\hat M^m(c)$ under the conditions:
\begin{enumerate}
	\item[(I)] $A\mathfrak JTM^\perp\subset \mathfrak JTM^\perp$;
	\item[(II)] $AJTM^\perp\subset JTM^\perp$.
\end{enumerate}

Recall that a Hopf hypersurface is a real hypersurface which satisfies the condition (II).
Let $\mathfrak D^\perp=\mathfrak JTM^\perp$. 
A real hypersurface $M$ is said to be $\mathfrak D^\perp$-invariant if it satisfies the condition (I).
The following theorems provide a list of  possible real hypersurfaces satisfying these two conditions.

\begin{theorem}[\cite{berndt-suh}]\label{thm:+}
Let $M$ be a connected real hypersurface in $SU_{m+2}/S(U_2U_m)$, $m\geq3$. 
Then $M$ is Hopf and $\mathfrak D^\perp$-invariant 
if and only if
one of the following holds:
\begin{enumerate}
\item[$(A)$] $M$ is an open part of a tube around a totally geodesic $SU_{m+1}/S(U_2U_{m-1})$ in $SU_{m+2}/S(U_2U_m)$, or
\item[$(B)$] $M$ is an open part of a tube around a totally geodesic $\mathbb  HP_n=Sp_{n+1}/Sp_1Sp_n$ in $SU_{m+2}/S(U_2U_m)$,
                  where  $m=2n$ is even. 
\end{enumerate}
\end{theorem}

\begin{theorem}[\cite{berndt-suh2}]\label{thm:-}
Let $M$ be a connected real hypersurface in $SU_{2,m}/S(U_2U_m)$, $m\geq2$
Then $M$ $M$ is Hopf and $\mathfrak D^\perp$-invariant
 if and only if
one of the following holds:
\begin{enumerate}
\item[$(A)$] $M$ is an open part of a tube around a totally geodesic $SU_{2,m-1}/S(U_2U_{m-1})$ in $SU_{2,m}/S(U_2U_m)$, or
\item[$(B)$] $M$ is an open part of a tube around a totally geodesic $\mathbb  HH_n=Sp_{1,n}/Sp_1Sp_n$ in $SU_{2,m}/S(U_2U_m)$,
                  where  $m=2n$ is even, or 
\item[$(C_1)$] $M$ is an open part of a horosphere in $SU_{2,m}/S(U_2U_m)$ whose center at infinity is singular and of 
type $JX\in  \mathfrak JX$, or
\item[$(C_2)$]	$M$ is an open part of  a horosphere in $SU_{2,m}/S(U_2U_m)$ whose center at infinity is singular and of type $JX\perp  \mathfrak JX$, 	or 							
\item[$(D)$] the normal vectors $N$  of $M$ at each point $x\in M$ is singular of type $JN\perp \mathfrak JN$. Moreover, $M$ has at least four distinct principal curvatures, which are given by
\[
\alpha=2\sqrt{-c}, \quad \gamma=0, \quad \beta=\sqrt{-c}, 
\]
with corresponding principal curvature spaces
\begin{align*}
&T_\alpha=JTM^\perp\oplus\mathfrak JTM^\perp,  
&T_\gamma=\mathfrak JJTM^\perp,  ~~
&T_\beta\perp JTM^\perp\oplus\mathfrak JTM^\perp\oplus\mathfrak JJTM^\perp.
\end{align*}
If $\mu$ is another (possibly nonconstant) principal curvature function, then 
$JT_\mu\subset T_\beta$ and 
$\mathfrak JT_\mu\subset T_\beta$. 
\end{enumerate}
\end{theorem}
Real hypersurfaces of type $A$, $B$, $C_1$ and $C_2$ in Theorem~\ref{thm:-} and that in  
Theorem~\ref{thm:+}
have been the main focus along this line in the past two decades.
Finding simplest conditions characterizing  real hypersurfaces in these theorems  (or any of its subclasses)  has become a key step in the study of 
such real hypersurfaces.

Observe that the unit normal vector field $N$ for real hypersurfaces $M$ appeared in these theorems  is singular.
In this line of thought, for a real hypersurface $M$ in $\hat M^m(c)$, Lee and Suh showed that if  $M$ is Hopf  and $N$ is singular of type 
$JN\perp\mathfrak JN$ everywhere, then it is $\mathfrak D^\perp$-invariant  (cf. \cite{lee-suh}, \cite{suh2}). 
On the other hand, for the case $c>0$, it was shown in \cite{lee-loo2} that the condition (II) is necessary for the condition (I) and $JN\in \mathfrak JN$
everywhere.

In this paper, we show that the condition (II) is unnecessary in these two theorems, that is, we show that 
\begin{theorem}\label{T1.5}
Let $M$ be a connected real hypersurface in $\hat M^m(c)$, $m\geq2$. If $M$ is $\mathfrak D^\perp$-invariant, then it satisfies the condition (II), that is, $M$ is Hopf.
\end{theorem}

It is worthwhile to remark that there is no known example for Case (D) in Theorem~\ref{thm:-}.
In \cite{berndt-suh}, Berndt and Suh conjectured that such a real hypersurface does not exist.
With the assumption that the principal curvatures being constant, 
we prove that the nonexistence of such real hypersurfaces and this  gives a partial answer to the conjecture. More precisely, we have 
\begin{theorem}\label{T2}
Let $M$ be a connected real hypersurface in $SU_{2,m}/S(U_2U_m)$, $m\geq2$. 
If $M$ is $\mathfrak D^\perp$-invariant  and has constant principal curvatures, then one of the following holds:
\begin{enumerate}
\item[$(A)$] $M$ is an open part of a tube around a totally geodesic $SU_{2,m-1}/S(U_2U_{m-1})$ in $SU_{2,m}/S(U_2U_m)$, or
\item[$(B)$] $M$ is an open part of a tube around a totally geodesic $\mathbb  HH_n=Sp_{1,n}/Sp_1Sp_n$ in $SU_{2,m}/S(U_2U_m)$,
                  where  $m=2n$ is even, or 
\item[$(C_1)$] $M$ is an open part of a horosphere in $SU_{2,m}/S(U_2U_m)$ whose center at infinity is singular and of 
type $JX\in  \mathfrak JX$, or
\item[$(C_2)$]	$M$ is an open part of  a horosphere in $SU_{2,m}/S(U_2U_m)$ whose center at infinity is singular and of type $JX\perp  \mathfrak JX$.
\end{enumerate}
\end{theorem}

\section{Preliminaries}
In this section, we recall some  fundamental identities for real hypersurfaces in  complex Grassmannian of rank two,
which  have been proven in \cite{berndt-suh, berndt-suh2, lee-loo, looth}. 

Let  $\hat M^m(c)$ be the compact complex Grassmannian $ SU_{m+2}/S(U_2U_m)$ of rank two 
 (resp.  noncompact complex Grassmannian $SU_{2,m}/S(U_2U_m)$ of  rank two) for $c>0$ (resp. $c<0$), 
where $c=\max||K||/8$ is a scaling factor for the Riemannian metric $g$ and $K$ is the sectional curvature for $\hat M^m(c)$.
The Riemannian geometry of $\hat M^m(c)$ were studied in \cite{berndt, berndt-suh, berndt-suh2}.
Denote by $J$ and $\mathfrak J$ the K\"ahler structure $J$ and quarternionic K\"ahler structure on $\hat M^m(c)$ respectively.

Let $M$ be a connected, oriented real hypersurface isometrically immersed in $\hat M^m(c)$, $m\geq 2$, 
and $N$ be a unit normal vector field on $M$. Denote by the same $g$ the Riemannian metric on $M$. 
 The almost contact metric 3-structure $(\phi_a,\xi_a,\eta_a,g)$ on $M$ is given by
$$
J_a X=\phi_a X+\eta_a (X) N,\quad\quad J_a N=-\xi_a,\quad\quad \eta_a(X)=g(X,\xi_a),
$$
for any $X\in TM$, where  $\{J_1,J_2,J_3\}$ is a  canonical local basis  of $\mathfrak J$ on $\hat M^m(c)$.
It follows that 
\begin{align*}
\phi_a\phi_{a+1}-\xi_a\otimes\eta_{a+1}=\phi_{a+2}
\\ 
\phi_a\xi_{a+1}=\xi_{a+2}=-\phi_{a+1}\xi_a  
\end{align*}
for $a\in\{1,2,3\}$. The indices in the preceding equations are taken modulo three. 

The  K\"ahler structure $J$ induces on $M$ an almost contact metric structure by 
\begin{align*}	
JX=\phi X+\eta(X)N,	\quad JN=-\xi, \quad \eta(X)=g(X,\xi). 
\end{align*}
Let $\mathfrak D^\perp=\mathfrak JTM^\perp$, and $\mathfrak D$ its orthogonal complement in $TM$. 
We define a local  $(1,1)$-tensor field $\theta_a$ on $M$  by
\[
\theta_a :=\phi_a\phi -\xi_a\otimes\eta.
\]
Denote by $\nabla$ the Levi-Civita connection on $M$. Then there exist local $1$-forms $q_a$, $a\in\{1,2,3\}$ such that 
\begin{eqnarray}\label{eqn:contact}
\left.\begin{aligned}
& \nabla_X \xi = \phi AX \\
&\nabla_X \xi_a = \phi_a AX+q_{a+2}(X)\xi_{a+1}-q_{a+1}(X)\xi_{a+2} \\
&\nabla_X\phi\xi_a
=\theta_aAX+\eta_a(\xi)AX+q_{a+2}(X)\phi\xi_{a+1} - q_{a+1}(X)\phi\xi_{a+2}. 
\end{aligned}\right\}
\end{eqnarray}
The following identities are known. 
\begin{lemma}[\cite{lee-loo}]
\label{lem:theta}
\begin{enumerate}
\item[(a)] $\theta_a$ is symmetric, 
\item[(b)] $\phi\xi_a=\phi_a\xi$,
\item[(c)] $\theta_a\xi=-\xi_a; \quad \theta_a\xi_a=-\xi; \quad \theta_a\phi\xi_a=\eta(\xi_a)\phi\xi_a$, 
\item[(d)]  $\theta_a\xi_{a+1}= \phi\xi_{a+2}=-\theta_{a+1}\xi_a$,
\item[(e)] 
$-\theta_a\phi\xi_{a+1}+\eta(\xi_{a+1})\phi\xi_a =\xi_{a+2}=\theta_{a+1}\phi\xi_a-\eta(\xi_a)\phi\xi_{a+1}$.
\end{enumerate}
\end{lemma}
\begin{lemma}[\cite{lee-loo}]\label{lem:Aphixi_a}
If $\xi\in\mathfrak D$ everywhere, then $A\phi\xi_a=0$ for $a\in\{1,2,3\}$.
\end{lemma}

For each $x\in M$,   we define a subspace $\mathcal H^\perp$ of $T_xM$ by
$$\mathcal H^\perp: =\mathrm{span}\{\xi,\xi_1,\xi_2,\xi_3,\phi\xi_1,\phi\xi_2,\phi\xi_3\}.$$
Let $\mathcal{H}$ be the orthogonal complement of $\mathcal {H}^\perp$ in $T_xM$. 
Then  
$\dim\mathcal H=4m-4$  (resp. $\dim\mathcal H=4m-8$) when $\xi\in\mathfrak D^\perp$ (reps. $\xi\notin\mathfrak D^\perp$).
Moreover, $\theta_{a|\mathcal{H}}$ has two eigenvalues: $1$ and $-1$. 
Denote by  $\mathcal H_a(\varepsilon)$ the eigenspace corresponding to the eigenvalue $\varepsilon$ of 
${\theta_a}_{|\mathcal H}$.
Then $\dim \mathcal H_a(1)=\dim \mathcal H_a(-1)$ is even, and 
\begin{align*}
\begin{aligned}
&\phi\mathcal H_a(\varepsilon)=\phi_a\mathcal H_a(\varepsilon)=\theta_a\mathcal H_a(\varepsilon)=\mathcal H_a(\varepsilon) \\
&\phi_b\mathcal H_a(\varepsilon)=\theta_b\mathcal H_a(\varepsilon)=\mathcal H_a(-\varepsilon), \quad (a\neq b).
\end{aligned}
\end{align*}

The equations of Gauss and Codazzi are respectively given by
$$\begin{aligned}
R(X,Y)Z=&g( AY,Z) AX-g( AX,Z) AY+c\{g( Y,Z) X-g(X,Z) Y\\
&+g(\phi Y,Z)\phi X-g(\phi X,Z)\phi Y -2g(\phi X,Y)\phi Z\}\\
&+c\sum_{a=1}^3\{g(\phi_aY,Z)\phi_aX-g(\phi_aX,Z) \phi_aY-2g(\phi_aX,Y)\phi_aZ\\
&+g(\theta_aY,Z)\theta_aX-g(\theta_aX,Z)\theta_aY\}
\end{aligned}$$
\begin{align*}
(\nabla_X A)Y-(\nabla_Y A)X=&c\{\eta(X)\phi Y-\eta(Y)\phi X-2g(\phi X,Y)\xi\}\\
&+c\sum_{a=1}^3 \{\eta_a(X)\phi_a Y-\eta_a(Y)\phi_a X -2g(\phi_a X,Y)\xi_a\\&
+\eta_a(\phi X)\theta_a Y-\eta_a(\phi Y)\theta_a X\}.
\end{align*}



\section{Proof of Theorem \ref{T1.5}}
Under the assumption that $A\mathfrak D^\perp\subset\mathfrak D^\perp$, after having  a suitable choice of canonical local basis $\{J_1,J_2,J_3\}$ for $\mathfrak  J$, the (local) vector fields $\xi_1,\xi_2,\xi_3$ are principal, say $A\xi_a=\beta_a\xi_a$, for $a\in\{1,2,3\}$.
By using the Codazzi equation of such real hypersurfaces, we have 
(cf. \cite{berndt-suh}, \cite{berndt-suh2})
\begin{equation}\label{Long}
\begin{aligned}
\beta_a & g((\phi_a A+   A\phi_a)X,Y)-2g(A\phi_a AX,Y)\\
	-	&(\beta_a-\beta_{a+1})q_{a+2}(\xi_a)\{\eta_a(X) \eta_{a+1}(Y)-\eta_{a+1}(X)\eta_a(Y)\}\\
	-&(\beta_a-\beta_{a+2})q_{a+1}(\xi_a)\{\eta_{a+2}(X)\eta_a(Y)-\eta_a(X)\eta_{a+2}(Y)\}\\
=&c\big(-2\eta(\xi_a)g(\phi X,Y)-2g(\phi_a X,Y)+2\eta(X)\eta_a(\phi Y)-2\eta(Y)\eta_a(\phi X)\\
  &+2\eta_{a+1}(X)\eta_{a+2}(Y)-2\eta_{a+1}(Y)\eta_{a+2}(X)	\\
	&+2\eta_{a+1}(\phi X)\eta_{a+2}(\phi Y)-2\eta_{a+1}(\phi Y)\eta_{a+2}(\phi X)\\
	& +2\eta_a (Y)\left\{2\eta(\xi_a)\eta_a(\phi X)-\eta(\xi_{a+1})\eta_{a+1}(\phi X)-\eta(\xi_{a+2})\eta_{a+2}(\phi X)\right\}\\
	& -2\eta_a (X)\left\{2\eta(\xi_a)\eta_a(\phi Y)-\eta(\xi_{a+1})\eta_{a+1}(\phi Y)-\eta(\xi_{a+2})\eta_{a+2}(\phi Y)\right\}\\
	&-(\beta_a-\beta_{a+1})\{q_{a+2}(X)\eta_{a+1}(Y)-q_{a+2}(Y)\eta_{a+1}(X)\}\\
	&+(\beta_a-\beta_{a+2})\{q_{a+1}(X)\eta_{a+2}(Y)-q_{a+1}(Y)\eta_{a+2}(X)\}\big)
\end{aligned}
\end{equation}
for all $X,Y$ tangent to $M$.

We consider two cases: 
(i) 	$\xi\in\mathfrak D^\perp$ everywhere, and 
(ii) $\xi\notin\mathfrak D^\perp$ at some points of $M$.

\textbf{Case (i): } Suppose $\xi\in\mathfrak D^\perp$ everywhere.
For  each $x\in M$, since $\mathfrak D$ is invariant under $A$, $\phi$ and $\phi_a$, from  (\ref{Long}), we have 
\begin{align*}
&2c\{\eta(\xi_a)\phi X+\phi_a X\}+\beta_a (\phi_a A+A\phi_a)X-2A\phi_a AX\\
	&=-(\beta_a-\beta_{a+1})q_{a+2}(X)\xi_{a+1}+(\beta_a-\beta_{a+2})q_{a+1}(X)\xi_{a+2}.
\end{align*}
for all $X\in\mathfrak D$ and $a\in\{1,2,3\}$. 
It follows that 
\begin{align}\label{Big}
2c\{\eta(\xi_a)\phi X+\phi_a X\}+\beta_a (\phi_a A+A\phi_a)X-2A\phi_a AX=0\\
(\beta_a-\beta_{a+1})q_{a+2}(X)=0 \notag
\end{align}
for all $X\in\mathfrak D$ and $a\in\{1,2,3\}$. 
Note that if $X\in\mathfrak D$, then $\phi_a X\in\mathfrak D$. Next, applying  $\phi_a$ on both sides of (\ref{Big}) and replacing $X$ by $\phi_a X$ in (\ref{Big}) give
\begin{align}\label{Big-2}
2c\{\eta(\xi_a)\theta_a X-X\}-\beta_a AX+\beta_a\phi_a A\phi_a X-2\phi_a A\phi_a AX=0
\end{align}
and
$$2c\{\eta(\xi_a)\theta_a X-X\}+\beta_a\phi_a A\phi_a X-\beta_a AX-2A\phi_a A\phi_a X=0$$
respectively.
Hence, 
\begin{align}\label{commute}
(\phi_a A\phi_a) AX=A(\phi_a A\phi_a) X, \quad a\in\{1,2,3\}
\end{align}
for all $X\in\mathfrak D$.
With the assumption that $\xi\in\mathfrak D^{\perp}$, we have $\mathfrak D=\mathcal{H}$ and $A\mathcal{H}\subset\mathcal{H}$. By 
 (\ref{commute}), there exist common orthonormal eigenvectors $X_1,\cdots, X_{4m-4}\in\mathcal{H}$ of $A$ and $\phi_1 A \phi_1$.
It follows that
$AX_j=\lambda_j X_j$  and  $ A\phi_1 X_j=\mu_j\phi_1 X_j$.
Using these in (\ref{Big-2}), we have
$$
2c\eta(\xi_1)\theta_1 X_j-(2c+\lambda_j\beta_1+\mu_j\beta_1-2\lambda_j\mu_j)X_j=0.
$$
Since $\xi\in\mathfrak D^{\perp}$, we may suppose $\eta(\xi_1)\neq 0$. Then
$$
\theta_1 X_j-\varepsilon X_j=0.
$$
where 
$\varepsilon=(2c+\beta_1(\lambda_j+\mu_j)-2\lambda_j\mu_j)/2c\eta(\xi_1)$.
It follows that  $\varepsilon\in\{1,-1\}$. Without loss of generality, we can assume that 
$$
X_1,\cdots, X_{2m-2}\in\mathcal{H}_1(1) \quad\quad\text{ and }\quad\quad X_{2m-1},\cdots, X_{4m-4}\in\mathcal{H}_1(-1).
$$
Consequently, 
$
A\mathcal{H}_1(1)\subset\mathcal{H}_1(1)
$
and hence
$
\phi_2 A\phi_2\mathcal{H}_1(1)\subset\mathcal{H}_1(1).
$
Thus, if we take $a=2$ in (\ref{commute}), then there exist orthonormal vectors 
$\tilde  X_1,\cdots, \tilde  X_{2m-2}\in\mathcal{H}_1(1)$ 
such that
$A\tilde  X_j=\tilde \lambda_j\tilde  X_j$ 
 and
$A\phi_2\tilde  X_j=\tilde \mu_j\phi_2\tilde  X_j$.
From (\ref{Big-2}), we have
$$
2c\eta(\xi_2)\theta_2 \tilde  X_j-(2c+\tilde \lambda_j\beta_2+\tilde \mu_j\beta_2-2\tilde \lambda_j\tilde \mu_j) \tilde  X_j=0.
$$
Since $\tilde  X_j\in\mathcal{H}_1(1)$  and $\theta_2\tilde  X_j\in\mathcal{H}_1(-1)$, we have
$
\eta(\xi_2)=0.
$
In a similar manner, we obtain $\eta(\xi_3)=0$. Thus, we have $\eta(\xi_1)=\pm 1$ or $\xi=\pm\xi_1$.
As a result, we have shown that $A\xi=\beta_1\xi$ at each $x\in M$. Hence $M$ is Hopf.

\medskip
\textbf{Case (ii): }
Suppose that $\xi\notin\mathfrak D^\perp$ at a point $x\in M$.
Since $\mathfrak D$ is invariant under $A$  and $\phi_a$,  after letting $Y\in\mathfrak D$ and $X=\xi_{a+1}$ in  (\ref{Long}), we have 
\begin{align}\label{eqn:I}
(\beta_a-\beta_{a+1})q_{a+2}(Y)=cg\big(2\eta_{a+1}(\xi)\phi\xi_a+4\eta_a(\xi)\phi\xi_{a+1}+2\eta_{a+2}(\xi)\xi,Y\big).
\end{align}
Similarly, if we let $Y\in\mathfrak D$ and $X=\xi_{a+2}$ in (\ref{Long}), then 
\begin{align*}
(\beta_{a+2}-\beta_a)q_{a+1}(Y)=cg\big(4\eta_a(\xi)\phi\xi_{a+2}+2\eta_{a+2}(\xi)\phi\xi_a-2\eta_{a+1}(\xi)\xi,Y\big).
\end{align*}
Raising the index of this equation by one gives
\begin{align}\label{eqn:II}
(\beta_a-\beta_{a+1})q_{a+2}(Y)=cg\big(4\eta_{a+1}(\xi)\phi\xi_a+2\eta_a(\xi)\phi\xi_{a+1}-2\eta_{a+2}(\xi)\xi,Y\big).
\end{align}
Denote by $P$ the projection from $T_xM$ onto $\mathfrak D$.
Since $Y$ is an arbitrary vector in $\mathfrak D$, by (\ref{eqn:I}) and (\ref{eqn:II}), we obtain
\[
P\left(2\eta_{a+2}(\xi)\xi-\eta_{a+1}(\xi)\phi\xi_a+\eta_a(\xi)\phi\xi_{a+1}\right)=0
\]
which implies that
\begin{align*}
2\eta_{a+2}(\xi)\xi &-\eta_{a+1}(\xi)\phi\xi_a+\eta_a(\xi)\phi\xi_{a+1}	\\
   &- \sum^3_{b=1}g\left(2\eta_{a+2}(\xi)\xi-\eta_{a+1}(\xi)\phi\xi_a+\eta_a(\xi)\phi\xi_{a+1},\xi_b\right)\xi_b=0.
\end{align*}
Since $\{\xi,\xi_1,\xi_2,\xi_3,\phi\xi_1,\phi\xi_2,\phi\xi_3\}$ is linear independent, $\eta_a(\xi)=0$ for $a\in\{1,2,3\}$
which means that $\xi\in\mathfrak D$ at $x\in M$.
Hence,  we conclude that $\xi\in\mathfrak D$ everywhere by the connectedness of $M$ and the continuity of $\sum^3_{a=1}\eta_a(\xi)^2$.

Since $\eta_a(\xi)=0$ for $a\in\{1,2,3\}$, the equations (\ref{Long}) and (\ref{eqn:I}) give
\begin{align*}
\begin{aligned}
\beta_a(\phi_a A+A\phi_a)X-2A\phi_a AX
=&2c\{-\phi_a X-\eta(X)\phi_a\xi-\eta(\phi_a X)\xi\\
	&-\eta_{a+1}(\phi X)\phi\xi_{a+2}+\eta_{a+2}(\phi X)\phi\xi_{a+1}\}
\end{aligned}
\end{align*}
for all $X\in\mathfrak D$.
First, acting $\phi_a$ on both sides of  the preceding equation and next  
replacing $X$ by $\phi_a X$ in  the preceding equation give
\begin{align}\label{Long-2}
\begin{aligned}
\beta_a(-A+\phi_a A\phi_a)X & -2\phi_a A\phi_a AX
=2c\{X+\eta(X)\xi-\eta(\phi_a X)\phi_a\xi\\
  & +\eta(\phi_{a+1} X)\phi_{a+1}\xi+\eta(\phi_{a+2}X)\phi_{a+2}\xi \}  
\end{aligned}
\end{align}
and 
\begin{align*}
\begin{aligned}
\beta_a(\phi_a A\phi_a-A)X & -2A\phi_a A\phi_a X
=2c\{X -\eta(\phi_a X)\phi_a\xi+\eta(X)\xi\\
  & +\eta(\phi_{a+2}X)\phi_{a+2}\xi +\eta(\phi_{a+1} X)\phi_{a+1}\xi\}  
\end{aligned}
\end{align*}
for all $X\in\mathfrak D$. 
Hence $(\phi_aA\phi_a)AX=A(\phi_aA\phi_a)X$ for $X\in\mathfrak D$.
Since $A\mathfrak D\subset\mathfrak D$ and $\phi_aA\phi_a\mathfrak D\subset\mathfrak D$, there exists an orthonormal basis 
$
\{X_1,\cdots,X_{4m-4}\}
$
for $\mathfrak D$ such that $AX_j=\lambda_jX_j$ and $A\phi_aX_j=\mu_j\phi_aX_j$.
Take an arbitrary $j\in\{1,2,\cdots,4m-4\}$ and substitute $X=X_j$ in (\ref{Long-2}), we obtain
\begin{align*}
\begin{aligned}
\{2\lambda_j\mu_j-\beta_a(\lambda_j+\mu_j)\}X_j=&
2c\{X_j+\eta(X_j)\xi-\eta(\phi_a X_j)\phi_a\xi\\
  & +\eta(\phi_{a+1} X_j)\phi_{a+1}\xi+\eta(\phi_{a+2}X_j)\phi_{a+2}\xi \} . 
\end{aligned}
\end{align*}
Let $X_j^H$ be  the $\mathcal H$-component of $X_j$. Then the $\mathcal H$- and $\xi$-components of the preceding equations are given respectively by 
\begin{align*}
\begin{aligned}
&0=\{2\lambda_j\mu_j-\beta_a(\lambda_j+\mu_j)-2c\}X_j^H\\
&0=\{2\lambda_j\mu_j-\beta_a(\lambda_j+\mu_j)-4c\}\eta(X_j).
\end{aligned}
\end{align*}
For each $j\in\{1,2,\cdots,4m-4\}$ with $X_j^H\neq0$, 
the former of these equations implies that  $2\lambda_j\mu_j-\beta_a(\lambda_j+\mu_j)=2c$. This, together with the latter, gives $\eta(X_j)=0$. This means that $A\xi\perp\mathcal H$.
Since $A\phi\xi_a=0$  for $a\in\{1,2,3\}$  by Lemma~\ref{lem:Aphixi_a},
 we obtain $\xi$ is principal on $M$. 
This completes the proof.


\section{Proof of Theorem~\ref{T2}}
It is clear that we only need to consider the case   $\xi\in\mathfrak D$ everywhere.
Denote the spectrum of $A_{|\mathcal H}$ by $\spec$.
For each $\lambda\in\spec$, denote by $T_{\lambda}$ the subbundle of $\mathcal H$ foliated by eigenspaces of  $A_{|\mathcal H}$
corresponding to $\lambda$.
We will use Cartan's method to prove the result, and begin with some fundamental identities.
\begin{lemma}\label{lem:4}
For any $\lambda$, $\mu\in\spec$,  if   $\lambda\neq\mu$, then
\begin{enumerate}
\item[(a)] $(\lambda-\mu)g(\nabla_XY,V)  =g((\nabla_XA)Y,V)$
\item[(b)] $\nabla_YZ~ \perp ~ T_{\mu}$
\item[(c)]   $g(\nabla_{[Y,V]}Y,V)=g(\nabla_YV,\nabla_VY)+c\left(g(\phi Y,V)^2
+\sum^3_{a=1}\{g(\phi_a Y,V)^2+g(\theta_a Y,V)^2\}\right)$
\end{enumerate}
for all vector fields $Y,Z$ tangent to $T_\lambda$, $V$ tangent to $T_{\mu}$ and $X\in TM$.
\end{lemma}
\begin{proof}
Statement  (a) is trivial. Next, by the Codazzi equation and (a), we have 
\begin{align*}
0=&g((\nabla_YA)V,Z)-g((\nabla_VA)Y,Z)\\
=&(\lambda-\mu)g(\nabla_YZ,V).
\end{align*}
This gives Statement (b).
Similarly,  with the help of the Codazzi equation and  (\ref{eqn:contact}), we compute
\begin{align*}
(\mu-&\lambda)g(\nabla_{[Y,V]}V,Y)\\
=&g((\nabla_{[Y,V]}A)V,Y)\\
=&g((\nabla_{\nabla_YV}A)V,Y)-g((\nabla_{\nabla_VY}A)Y,V)\\
=&g((\nabla_VA)Y,\nabla_YV)-g((\nabla_YA)V,\nabla_VY)\\
   &+c\left(\eta(\nabla_YV)g(\phi V,Y)+\sum^3_{a=1}\{
\eta_a(\nabla_YV)g(\phi_a V,Y)+\eta(\phi_a\nabla_YV)g(\theta_a V,Y)\}\right) \\
   &-c\left(\eta(\nabla_VY)g(\phi Y,V)+\sum^3_{a=1}\{
\eta_a(\nabla_VY)g(\phi_a Y,V)+\eta(\phi_a\nabla_VY)g(\theta_a Y,V)\}\right) \\
=&(\lambda-\mu)g(\nabla_VY,\nabla_YV)
     +(\lambda-\mu)c\left(g(\phi Y,V)^2
+\sum^3_{a=1}\{g(\phi_a Y,V)^2+g(\theta_a Y,V)^2\}\right).
\end{align*}
Hence we obtain Statement (c).
\end{proof}

 If $\sqrt{-c}\notin\spec$, then $M$ is an open part of real hypersurfaces of type $A$, $B$ and $C_1$.
Next, consider the case  $\sqrt{-c}\in\spec$. We claim that $\spec=\{\sqrt{-c}\}$.
Suppose to the contrary that $\spec\neq\{\sqrt{-c}\}$.
Let $\dim T_{\sqrt{-c}}=4m-8-p$, $p>0$ and $\{X_1,\cdots,X_p\}$ be a local orthonormal frame on $\mathcal H$ with $AX_j=\lambda_jX_j$,
where $\lambda_j\neq\sqrt{-c}$ for each $j\in\{1,\cdots,p\}$.
Take a unit vector field $E$ tangent to $T_{\sqrt{-c}}$, we have

\begin{align*}
&c\left(3g(\phi E,X_j)^2+\sum^3_{a=1}\{3g(\phi_a E,X_j)^2-g(\theta_aE,X_j)^2
+g(\theta_aE,E)g(\theta_aX_j,X_j)\}\right)\\
&+c+\lambda_j\sqrt{-c}\\
=&g(R(X_j,E)E,X_j)\\
=&g(\nabla_{X_j}\nabla_EE,X_j)-g(\nabla_E\nabla_{X_j}E,X_j)-g(\nabla_{[X_j,E]}E,X_j)\\
=&-g(\nabla_EE,\nabla_{X_j}X_j)+g(\nabla_{X_j}E,\nabla_EX_j)-g(\nabla_{[X_j,E]}E,X_j)\\
=&-\lambda\sqrt{-c}\sum^3_{a=1}g(\theta_aE,E)g(\theta_aX_j,X_j)+2g(\nabla_{X_j}E,\nabla_EX_j)\\
   &+c\left(g(\phi E,X_j)^2+\sum^3_{a=1}\{g(\phi_aE,X_j)^2+g(\theta_aE,X_j)^2\}\right)\\
=&-\lambda\sqrt{-c}\sum^3_{a=1}g(\theta_aE,E)g(\theta_aX_j,X_j)
			+2\sum^p_{\stackrel{k=1}{\lambda_k\neq\lambda_j}}
			g(\nabla_{X_j}E,X_k)g(\nabla_EX_j,X_k)						\\
   &-2\lambda\sqrt{-c}\left(g(\phi E,X_j)^2+\sum^3_{a=1}\{g(\phi_aE,X_j)^2-g(\theta_aE,X_j)^2\}\right)\\
	 &+c\left(g(\phi E,X_j)^2+\sum^3_{a=1}\{g(\phi_aE,X_j)^2+g(\theta_aE,X_j)^2\}\right).
\end{align*}
It follows that 
\begin{align*}
(c+&\lambda_j\sqrt{-c})\Big(1+2(g(\phi E,X_j)^2 \\
&+\sum^3_{a=1}\{2g(\phi_a E,X_j)^2-2g(\theta_aE,X_j)^2
+g(\theta_aE,E)g(\theta_aX_j,X_j)\}\Big)\\
=&2\sum^p_{\stackrel{k=1}{\lambda_k\neq\lambda_j}}
		g(\nabla_{X_j}E,X_k)g(\nabla_EX_j,X_k)	\\
=&2\sum^p_{\stackrel{k=1}{\lambda_k\neq\lambda_j}}
			\frac{g((\nabla_{X_j}A)E,X_k)}{\sqrt{-c}-\lambda_k}
											\frac{g(\nabla_EA)X_j,X_k)}{\lambda_j-\lambda_k}\\
=&2\sum^p_{\stackrel{k=1}{\lambda_k\neq\lambda_j}}
		\frac{g((\nabla_{X_k}A)E,X_j)^2}{(\sqrt{-c}-\lambda_k)(\lambda_j-\lambda_k)}.
\end{align*}
By applying the Codazzi equation and the preceding equation, we have

\begin{align}\label{eqn:m}
0=&2\sum^p_{j=1}\sum^p_{\stackrel{k=1}{\lambda_k\neq\lambda_j}}
		\frac{g((\nabla_{X_k}A)E,X_j)^2}{(\lambda_j-\sqrt{-c})(\sqrt{-c}-\lambda_k)(\lambda_j-\lambda_k)}   \notag\\
=&
\sum^p_{j=1}\frac{c+\lambda_j\sqrt{-c}}{\lambda_j-\sqrt{-c}}
			\Big(1+2(g(\phi E,X_j)^2																\notag\\
			&+\sum^3_{a=1}\{2g(\phi_a E,X_j)^2-2g(\theta_aE,X_j)^2
			+g(\theta_aE,E)g(\theta_aX_j,X_j)			\}\Big) \notag\\
=&\sum^p_{j=1}\Big(1+2(g(\phi E,X_j)^2						\notag\\
			&+\sum^3_{a=1}\{2g(\phi_a E,X_j)^2-2g(\theta_aE,X_j)^2+g(\theta_aE,E)g(\theta_aX_j,X_j)
\}\Big).
\end{align}
By Theorem~\ref{thm:-}, we know that $\phi T_{\lambda_j}\subset T_{\sqrt{-c}}$ and $\phi_aT_{\lambda_j}\subset T_{\sqrt{-c}}$.
If  we take $E=\phi X_1$, then $\theta_a E=-\phi_aX_1$ is tangent to $T_{\sqrt{-c}}$. 
Hence,  the equation (\ref{eqn:m}) reduces to
\begin{align}\label{eqn:m2}
0=&p+2+\sum^p_{j=1}\sum^3_{a=1}\left\{2g(\theta_a X_1,X_j)^2+g(\theta_a X_1,X_1)g(\theta_aX_j,X_j)\right\}.
\end{align}
Fixed $b\in\{1,2,3\}$, by substituting $E=\phi_bX_1$ in (\ref{eqn:m}), we obtain
\begin{align*}
0=&p+2+\sum^p_{j=1}\Big\{2g(\theta_bX_1,X_j)^2-2g(\theta_{b+1}X_1,X_j)^2-2g(\theta_{b+2}X_1,X_j)^2\\
&+g(\theta_{b}X_1,X_1)g(\theta_{b}X_j,X_j)\\
&-g(\theta_{b+1}X_1,X_1)g(\theta_{b+1}X_j,X_j)-g(\theta_{b+2}X_1,X_1)g(\theta_{b+2}X_j,X_j)
\Big\}.
\end{align*}
By summing up over $b$, we obtain
\begin{align}\label{eqn:m3}
0=&3p+6-\sum^p_{j=1}\sum^3_{b=1}\Big\{2g(\theta_bX_1,X_j)^2+g(\theta_{b}X_1,X_1)g(\theta_{b}X_j,X_j)
\Big\}.
\end{align}
Summing up (\ref{eqn:m2}) and (\ref{eqn:m3}) gives $4p+8=0$; a contradiction. Accordingly, $\spec=\{\sqrt{-c}\}$ and hence $M$ is an open part of real hypersurfaces of type $C_2$.

\end{document}